\documentclass[12pt]{amsart}
\usepackage{amssymb, enumerate, mdwlist, amscd}

\evensidemargin\paperwidth
\advance\evensidemargin-\textwidth
\advance\oddsidemargin-1in
\evensidemargin\oddsidemargin

\topmargin\paperheight
\advance\topmargin-\textheight
\advance\topmargin-1in

\theoremstyle{plain}
\newtheorem{theorem}{Theorem}[section]
\theoremstyle{plain}
\newtheorem{proposition}[theorem]{Proposition}
\theoremstyle{plain}
\newtheorem{lemma}[theorem]{Lemma}
\theoremstyle{plain}
\newtheorem{corollary}[theorem]{Corollary}
\theoremstyle{plain}

\theoremstyle{definition}
\newtheorem{definition}[theorem]{Definition}
\theoremstyle{remark}

\theoremstyle{remark}

\theoremstyle{remark}

\title[Group approximation in Cayley topology, III]
{Group approximation in Cayley topology and coarse geometry, \\ Part III: Geometric property (T).}
\author[M. Mimura]{Masato Mimura}
\address{MM: Mathematical Institute, Tohoku University}
\author[N. Ozawa]{Narutaka Ozawa}  
\address{NO: Research Institute for Mathematical Sciences, Kyoto University}
\author[H. Sako]{Hiroki Sako}
\address{HS: School of Science, Tokai University}
\author[Y. Suzuki]{Yuhei Suzuki}
\address{YS: Department of Mathematical Sciences, University of Tokyo}
\date{\today}


\begin{document}

\begin{abstract}
In this series of papers, we study correspondence between the following:
(1) large scale structure of the metric space $\bigsqcup_m \mathrm{Cay} \left( G^{(m)} \right)$ 
consisting of Cayley graphs of finite groups with $k$ generators;
(2) structure of groups which appear in the boundary of the set $\left\{ G^{(m)} \right\}$ 
in the space of $k$-marked groups. In this third part of the series, 
we show the correspondence among the metric properties `geometric property (T),' 
`cohomological property (T),' and the group property `Kazhdan's property (T).' 
Geometric property (T) of Willett--Yu is stronger than being expander graphs. 
Cohomological property (T) is stronger than geometric property (T) for general coarse spaces. 
\end{abstract}


\maketitle

\section{Introduction}
In 1967, D. Kazhdan introduced the concept of \textit{property} $(\mathrm{T})$ for locally compact groups 
in terms of uniform spectral gaps for all unitary representations (in this paper, we regard 
Proposition~\ref{proposition: (T)} as a definition of property $(\mathrm{T})$ for discrete groups), 
which represents \textit{extreme rigidity} of groups. See a book of Bekka--de la Harpe--Valette \cite{BookBekkadelaHarpeValette} 
for comprehensive treatise on this property. For instance, G.~Margulis has observed that for a residually 
finite and finitely generated group $G$ with property $(\mathrm{T})$, 
any box space $\mathop{\Box} G$ forms a family of \textit{expanders}, 
namely, a family of uniformly locally finite and finite connected graphs 
whose combinatorial Laplacians have the first positive eigenvalues bounded away from zero.
(On expanders, we refer the reader to a book \cite{BookLubotzky} by A.~Lubotzky). 
Here for such $G$ and a sequence of normal subgroups $N_1 > N_2 > \cdots $ of $G$ with finite indices 
with $\bigcap_{m} N_m= \{1_G\}$, the \textit{box space} $\mathop{\Box}_{\{N_m\}_m} G$ 
\textit{associated with} $\{N_m\}_m$ is the \textit{coarse disjoint union} 
(see Subsection~\ref{subsection: marked groups and ms}) of finite Cayley graphs $\mathrm{Cay}(G/N_m,S)$, 
where $S$ is a fixed finite generating set of $G$ 
(the coarse structure of the box space does not depend on the choice of $S$).
Expander sequence represents strong rigidity, and serves a counterexample of the surjective side 
of the coarse Baum--Connes conjecture for coarse spaces. 
We refer the reader to the monographs \cite{NowakYuBook} and \cite{RoeLectureNote} 
for the basics of this subject. 

We however may obtain expander sequence from a group \textit{far from} having property $(\mathrm{T})$. 
For instance, A.~Selberg has showed that for a concrete example of $\{N_m\}_m$ for $F_2$, 
the free group of rank $2$, the box space $\mathop{\Box}_{\{N_m\}_m} F_2$ forms an expander sequence 
(for details, see a forthcoming book \cite{BookLubotzkyZuk} of Lubotzky and \.{Z}uk on property $(\tau)$). 
It had been paid strong attention to the problem of whether one can distinguish expanders coming 
from property $(\mathrm{T})$ groups from ones coming from non-$(\mathrm{T})$ groups in terms of coarse 
geometric properties. This problem is related to a question by J.~Roe \cite{RoeLectureNote} to 
define `coarse property $(\mathrm{T})$.'

R.~Willett and G.~Yu \cite{WillettYuI}, \cite{WillettYuII} have studied the maximal coarse Baum--Connes 
conjecture and introduced the notion of \textit{geometric property }$(\mathrm{T})$ for a (coarse) 
disjoint union of uniformly locally finite and finite graphs, which is stronger than being an expander sequence. 
They have showed that this property is an obstruction to the surjectivity of the maximal coarse Baum--Connes 
assembly map, and that a box space $\mathop{\Box} G$ has geometric property $(\mathrm{T})$ if and only if 
$G$ possesses property $(\mathrm{T})$. In the later work \cite{WillettYu13}, they have extended 
the definition of geometric property $(\mathrm{T})$ for (weakly) monogenic coarse space of bounded geometry, 
and proved that this property is a coarse invariant. In this manner, they give a satisfactory answer to the 
problem mentioned above (on the other hand, the Selberg expander $\mathop{\Box}_{\{N_m\}}F_2$ is showed 
in \cite{ChWW} to admit a \textit{fibered coarse embedding} into a Hilbert space in the sense of 
Chen--Wang--Yu \cite{ChWY}, and this ensures the maximal coarse Baum--Connes conjecture for this space, 
see also \cite{WillettYuII}). 

It is a well-known theorem of Delorme--Guichardet (Theorem 2.12.4 in \cite{BookBekkadelaHarpeValette}) 
that property $(\mathrm{T})$ can be characterized in terms of $1$-cohomology with coefficients 
in unitary representations. In Section~\ref{section:cohomologicalT}, we investigate an analogue of 
this characterization in the setting of coarse geometry, and introduce 
\emph{cohomological property $(\mathrm{T})$} for coarse spaces. It will be proved that 
cohomological property $(\mathrm{T})$ implies geometric property $(\mathrm{T})$ (but not vice versa).

The goal of this paper is to provide a characterization of the (coarse) disjoint 
union $X:=\bigsqcup_m \mathrm{Cay}\left( G^{(m)}, s_1^{(m)}, s_2^{(m)}, \ldots, s_k^{(m)} \right)$ 
of finite Cayley graphs to enjoy geometric property $(\mathrm{T})$. In the previous works of the 
first-named and the third-named author, we have revealed that the concept of the 
\textit{space of marked groups} and \textit{Cayley topology} play a key r\^{o}le in studying 
coarse geometric properties for such $X$. More precisely, for the \textit{Cayley boundary} 
$\partial_{\mathrm{Cay}}(\{G^{(m)}\}_m):=\overline{\left\{ G^{(m)} \right\}_m}^{\mathrm{Cayley}} \setminus \left\{ G^{(m)} \right\}_m $ 
of a sequence $\left\{\left( G^{(m)}, s_1^{(m)}, s_2^{(m)}, \ldots, s_k^{(m)} \right)\right\}_m$ 
in the space $\mathcal{G}(k)$ of $k$-marked groups, the following holds 
(for details, we refer the reader to the corresponding papers):
\begin{enumerate}[$(i)$]
   \item \cite{MimuraSakoI}: $X$ as above has property $\mathrm{A}$\quad $\Leftrightarrow$ \quad $\partial_{\mathrm{Cay}}(\{G^{(m)}\}_m)$ is uniformly amenable\quad $\Leftrightarrow$ \quad every member of $\partial_{\mathrm{Cay}}(\{G^{(m)}\}_m)$ is amenable; 
   \item \cite{MimuraSakoII}: $X$ as above admits a fibered coarse embedding into a Hilbert space \quad $\Leftrightarrow$ \quad $\partial_{\mathrm{Cay}}(\{G^{(m)}\}_m)$ is uniformly a-$\mathrm{T}$-menable.
\end{enumerate}
For the definition of $\mathcal{G}(k)$, and the Cayley topology, see Subsection~\ref{subsection: Cayley topology}. We note that the two results above can be regarded as generalization of previously known results (respectively by E.~Guentner and Chen--Wang--Wang \cite{ChWW}) for box spaces. Indeed, in the box space case, for a fixed finite generating set $S=(s_1,\ldots ,s_k)$ of $G$, the sequence $\{(G/N_m,S)\}_m$ converges to $(G,S)$ in the Cayley topology and hence the singleton $\{(G,S)\}$ is the Cayley boundary of that sequence.

With the notation above, we shall state our main theorem in this paper, 
which generalizes the above-mentioned result of Willett and Yu for the box spaces. 
\begin{theorem}\label{theorem:main}
Let $\left\{ G^{(m)} \right\}_{m \in \mathbb{N}} 
= \left\{ \left( G^{(m)}, s_1^{(m)}, s_2^{(m)}, \ldots, s_k^{(m)} \right) \right\}_{m \in \mathbb{N}}$ 
be a sequence of finite $k$-marked groups, and 
$X=\bigsqcup_{m \in \mathbb{N}} \mathrm{Cay} \left(G^{(m)}, s_1^{(m)}, \ldots, s_k^{(m)} \right)$
be their disjoint union. Then, the following are equivalent.
\begin{enumerate}
\item\label{cond:maincaybdy}
Every member of $\partial_{\mathrm{Cay}}(\{G^{(m)}\}_m)$ has property $\mathrm{(T)}$ of Kazhdan.
\item\label{cond:maingeot}
The metric space $X$ has geometric property $\mathrm{(T)}$.
\item\label{cond:maincoht}
The metric space $X$ has cohomological property $\mathrm{(T)}$.
\end{enumerate}
\end{theorem}

Theorem~\ref{theorem:main} provides us with the following corollary, which generalizes the result of Margulis mentioned above:

\begin{corollary}\label{corollary:expander}
Let $\left\{ G^{(m)} \right\}_{m \in \mathbb{N}} 
= \left\{ \left( G^{(m)}, s_1^{(m)}, s_2^{(m)}, \ldots, s_k^{(m)} \right) \right\}_{m \in \mathbb{N}}$ 
be a sequence of finite $k$-marked groups. 
If every member of $\partial_{\mathrm{Cay}}(\{G^{(m)}\}_m)$ has property $\mathrm{(T)}$, then the sequence of Cayley graphs $\left\{\mathrm{Cay}\left( G^{(m)}, s_1^{(m)}, s_2^{(m)}, \ldots, s_k^{(m)} \right)\right\}_m$ forms an expander family.
\end{corollary}

In fact, the proof of Corollary~\ref{corollary:expander} does not require 
the notion of geometric property $\mathrm{(T)}$ and follows directly from 
the intermediate result, Proposition~\ref{Proposition: Uniform Kazhdan}. 
See also \cite{MimuraSakoIV} for a quantitative version of Corollary~\ref{corollary:expander}. 
The proof of Theorem~\ref{theorem:main} is scattered in this paper: 
$(\ref{cond:maingeot})\Rightarrow(\ref{cond:maincaybdy})$ is proved in Section~\ref{sec:geomTtoT},
the converse in Section~\ref{Section:TtogeomT},
and $(\ref{cond:maingeot})\Leftrightarrow(\ref{cond:maincoht})$ in Section~\ref{section:cohomologicalT}.
In these proofs, we have avoided the technical aspects of $\mathrm{C}^*$-algebra theory. 
Instead, they are organized at the end of this paper (Section~\ref{section:reptheory}), 
where we study the structure of the maximal uniform Roe algebra in detail. 
%In Appendix, we introduce and study the bounded cohomology groups of coarse spa%ces. As an application, we prove that the class ${\mathcal C}_\mathrm{reg}^1$ o%f
%countable groups $G$ with non-zero bounded cohomology groups
%$H_{\mathrm{b}}^2(G,\ell_1(G))$ is closed under the coarse equivalence and
%the measure equivalence. 
%The class ${\mathcal C}_\mathrm{reg}^1$ contains, for example, all
%\emph{acylindrically hyperbolic} groups in the sense of D. Osin
%\cite{OsinAcylindrically}. 

\subsection*{Acknowledgments}
This work gained momentum when the authors attended the conference 
``Metric Geometry and Analysis'' held at Kyoto University in December 2013, 
and was finished during the first and second-named authors were visiting 
Institut Henri Poncar\'e in the spring 2014 for the trimester on 
``Random Walks and Asymptotic Geometry of Groups.'' We gratefully 
acknowledge their hospitality and stimulating environments. 
The authors benefitted from various comments and suggestions by Romain Tessera and Rufus Willett. 
The authors would like to thank Shin-ichi Oguni for providing us with literature 
and information on (co)homology theory in coarse geometry. 
The first-named author was supported by JSPS Grant-in-Aid for Young Scientists (B), no.~25800033.
The second-named author was supported by JSPS KAKENHI Grant Number 23540233.
The fourth-named author was supported by JSPS Research Fellow, no.~25-7810, and 
the MEXT Program of Leading Graduate Schools. 

\section{Preliminaries}

\subsection{\bf Coarse equivalence}\label{subsection:coarseequivalence}

Recall from Definition 1.8 in \cite{RoeLectureNote}
that a map $f\colon X\to Y$ between metric spaces is said to be 
\emph{uniformly bornologous} if $\sup\{ d(f(x),f(y)) \mid d(x,y)\le R\}<\infty$ 
for every $R>0$. Two maps $f_i\colon X\to Y$ are \emph{close} if 
$\sup_x d(f_1(x),f_2(x))<\infty$. 
The two metric spaces $X$ and $Y$ are \emph{coarsely equivalent} 
if there are uniformly bornologous maps $f\colon X\to Y$ and $g\colon Y\to X$ 
such that $g\circ f$ and $f\circ g$ are close to $\mathrm{id}_X$ and $\mathrm{id}_Y$, 
respectively. 
Thus, for every $n\in{\mathbb N}$, 
the $n$-point metric space $\mathbf{n}=\{1,\ldots,n\}$ (say, $d(i,j)=|i-j|$) 
is coarsely equivalent to a point, and the metric spaces $Y$ and $Y\times\mathbf{n}$ 
(say, $d((y,i),(z,j))=d(y,z)+|i-j|)$ are coarsely equivalent, via the inclusion 
$\iota\colon Y\hookrightarrow Y\times\{1\}$ 
and the projection $\mathrm{pr}\colon Y\times\mathbf{n}\to Y$. 
Every coarse equivalence roughly arises in this way. 

\begin{lemma}\label{lemma:coarseequivstr}
If $f\colon X\to Y$ is a coarse equivalence between uniformly locally finite metric spaces $X$ and $Y$, 
then there are $n\in{\mathbb N}$ and an injective and uniformly bornologous map $\tilde{f}\colon X \to Y\times\mathbf{n}$ 
such that $f=\mathrm{pr}\circ\tilde{f}$.
\end{lemma}

\begin{proof}
Since $\{f^{-1}(y)\}_y$ are uniformly bounded subsets of a uniformly locally finite metric space $X$,
a matching theorem yields a finite partition $X=\bigsqcup_{i=1}^n X_i$ such that $f$ is injective on each of 
$X_i$'s. Define $\tilde{f}(x)=(f(x),i)$ for $x\in X_i$ and we are done.
\end{proof}

\subsection{\bf Cayley topology}\label{subsection: Cayley topology}
The space $\mathcal{G}(k)$ of $k$-marked groups and the Cayley topology on it 
enable us to regard a finitely generated group as a point in a topological space.

Let $(G, s_1, s_2, \ldots, s_k)$ be a $(k + 1)$-tuple of a group $G$ and its generators $s_1, \ldots s_k$.
We call such an object a \textit{$k$-marked group}.
Throughout this paper, we use the following terminologies:
\begin{itemize}
\item
A $k$-marked group $(H, s'_1, s'_2, \ldots, s'_k)$ is called a {\it quotient} of $(G, s_1, s_2, \ldots, s_k)$ 
if there exists a group homomorphism $\phi \colon G \to H$
satisfying that $\phi(s_j) = s'_j$ for all $j = 1, \ldots, k$.
Note that every member of $\mathcal{G}(k)$ is a quotient of $(F_k, a_1, a_2, \ldots, a_k)$, 
where $F_k$ is the free group generated by $a_1, \ldots, a_k$.
\item
If the above homomorphism $\phi$ is isomorphic, then
two $k$-marked groups $(G, s_1, s_2, \ldots, s_k)$ and $(H, s'_1, s'_2, \ldots, s'_k)$ are said to be {\it isomorphic}.
\item
A $k$-marked group $(G, s_1, s_2, \ldots, s_k)$ is said to be {\it finitely presented} if there exists a finite subset $A \subset F_k$
such that the minimal normal subgroup $N \subset F_k$ containing $A$ is the kernel of the quotient map $\phi \colon (F_k, a_1, \ldots, a_k) \to (G, s_1, \ldots, s_k)$.
\end{itemize}
We note that the finite presentability of a group is independent of the choices of markings (see V.2 in \cite{delaHarpeGGT}). 
We denote by $\mathcal{G}(k)$ the set of the isomorphism classes of $k$-marked groups 
and call it the \emph{space of $k$-marked group} or the Cayley topological space.
A natural topology on $\mathcal{G}(k)$ was introduced by Grigorchuk in \cite{GrigorchukIntermediate}.
We call it the \textit{Cayley topology}.
For details, see the book \cite[Section V.10]{delaHarpeGGT} of de la Harpe.
We also recall the notations used in the part I paper \cite{MimuraSakoI} 
by the first-named and the third-named authors.
The topology is generated by relations. More precisely,
it is generated by closed and open subsets
\[O \left( s_{j(1)}^{\epsilon(1)} s_{j(2)}^{\epsilon(2)} \cdots s_{j(n)}^{\epsilon(n)} \right)
= \left\{ (G, s_1, s_2, \ldots, s_k) \in \mathcal{G}(k) \bigm| \ s_{j(1)}^{\epsilon(1)} s_{j(2)}^{\epsilon(2)} \cdots s_{j(n)}^{\epsilon(n)} = 1_G \right\}.
\]
Here, $s_{j(l)}$ is one of the generators $\{s_1, s_2, \ldots, s_k\}$ and $\epsilon(l)$ is nothing or $-1$.
The resulting topology is known to be Hausdorff (in fact metrizable) and compact.
Two $k$-marked groups are close if the balls with large radius of their Cayley graphs are identical.
We may regard the following as a definition.

\begin{lemma}[Lemma 2.1 in \cite{MimuraSakoII}]
\label{lemma: partial isomorphism}
Let $G = (G, s_1, \ldots, s_k)$ be a $k$-marked group and $R$ be a constant greater than $1$.
Let $N(G, R)$ be the subset of $\mathcal{G}(k)$ consisting of $k$-marked groups $(H, s'_1, \ldots, s'_k)$ satisfying that
there exists a bijection 
$\phi \colon B(1_H, 2R) \to B(1_{G}, 2R)$ from the ball of $H$ to that of $G$ with radius $2R$ such that
\begin{itemize}
\item
$\phi(s'_j) = s_j$, for every $j = 1, \ldots, k$,
\item
$\phi(g^{-1}) = \phi(g)^{-1}$, for every $g \in B(1_H, 2R)$,
\item
$\phi(gh) = \phi(g)\phi(h)$, for every $g, h \in B(1_H, R)$.
\end{itemize}
The subsets $\{N(G, R)\}_R$ form a neighborhood system of $G \in \mathcal{G}(k)$. 
\end{lemma}
We call $\phi$ satisfying the three conditions in the lemma a partial isomorphism.

\begin{lemma}\label{lemma: quotients form clopen subset}
Let $\Gamma = (\Gamma, \sigma_1, \ldots, \sigma_k) \in \mathcal{G}(k)$ be a $k$-marked group. 
Assume that $\Gamma$ is finitely presented.
The set of all the quotient groups of $\Gamma$,
\[
{\mathcal Q}_\Gamma = \left\{(G, s_1, \ldots, s_k) \in \mathcal{G}(k) \mid 
\mbox{\normalfont $\exists$ surjective hom $\phi \colon \Gamma \to G$ s.t.\ $\phi(\sigma_j) = s_j$} \right\}, 
\]
is a closed and open subset of $\mathcal{G}(k)$.
\end{lemma}

\begin{proof}
For each relation 
$s_{j(1)}^{\epsilon(1)} s_{j(2)}^{\epsilon(2)} \cdots s_{j(n)}^{\epsilon(n)}$ defining $\Gamma$, 
the subset 
\[ O \left( s_{j(1)}^{\epsilon(1)} s_{j(2)}^{\epsilon(2)} \cdots s_{j(n)}^{\epsilon(n)} \right)  \subset \mathcal{G}(k)\]
is open and closed.
So is the finite intersection of such subsets corresponding to the relations in $\Gamma$.
The intersection is nothing other than ${\mathcal Q}_\Gamma$.
\end{proof}

Finally, we record the following fact. 
For the proof, take $\Gamma_R$ to be the marked group quotient of $F_k$ 
by all the relations of $G$ that have length at most $R$.
Note that when $G$ is finitely presented, $\Gamma_R$ eventually coincides with $G$.

\begin{lemma}\label{lemma:convergence from the above}
For every $G = (G, s_1, \ldots, s_k) \in \mathcal{G}(k)$,
there exists a sequence of finitely presented groups 
$\left\{  \Gamma_R = \left( \Gamma_R, s^{(R)}_1, \ldots, s^{(R)}_k \right) \right\}_{R \in \mathbb{N}}$ converging to $G$
such that
\begin{itemize}
\item
$G$ is a quotient of $\Gamma_R$ for every $R$,
\item
$\Gamma_{R + 1}$ is a quotient of $\Gamma_R$ for every $R$.
\end{itemize}
\end{lemma}

\subsection{Marked groups as metric spaces}\label{subsection: marked groups and ms}

Let $(X, d)$ be a metric space (for which we allow $d$ to take the value $\infty$).
The space $X$ is said to be \textit{uniformly locally finite} (or to have \textit{bounded geometry}) 
if $\sup_{x \in X} \sharp(B(x, R)) < \infty$ for every $R > 0$.
Throughout this paper, $B(x, R)$ denotes the closed ball with radius $R$ whose center is $x$.

Marked groups have provided interesting examples in metric geometry.
For a $k$-marked group $(G, s_1, \ldots, s_k)$,
the word metric $d \colon G \times G \to [0, \infty)$
is defined as follows
\[
d(g, h) = 
\min \left\{n \Bigm| 
\begin{array}{c}
\exists n, \exists j(1) ,\ldots, j(n), \exists \epsilon(1), \ldots, \epsilon(n) \in \{1, -1\},\\
g h^{-1} = s_{j(1)}^{\epsilon(1)} s_{j(2)}^{\epsilon(2)} \cdots s_{j(n)}^{\epsilon(n)}
\end{array}
\right\} 
\]
(or zero if $g=h$).
Note that we are using the right-invariant word metric. 
In this way, the group $G$ becomes a uniformly locally finite metric space.
(It is the vertex set of the Cayley graph equipped with the edge metric.) 
We denote this metric space by $\mathrm{Cay}(G) = \mathrm{Cay}(G, s_1, \ldots, s_k)$.
The coarse equivalence class of the metric on $\mathrm{Cay}(G)$ is independent of 
the choice of generating subset $\{s_1,\ldots,s_k\}$. 

The subject of this paper is a metric space of the form
\[
X=\bigsqcup_{m \in \mathbb{N}} \mathrm{Cay} \left(G^{(m)}, s_1^{(m)}, \ldots, s_k^{(m)} \right).
\]
where $\left\{\left(G^{(m)}, s_1^{(m)}, \ldots, s_k^{(m)} \right)\right\}_m$ is a sequence 
of finite $k$-marked groups. 
Note that the coarse structure of $X$ does depend on the choice of 
a family $\{ s_1^{(m)}, \ldots, s_k^{(m)} \}_m$ of generating subsets. 
For a disjoint union $X=\bigsqcup_m X^{(m)}$ of metric spaces, 
it is customary in coarse geometry to put a metric $d$ on $X$ in such a way that 
it coincides with the original metric on each of $X^{(m)}$'s and that 
$d(X^{(m)},X^{(n)})\to\infty$ as $|m-n|(m+n)\to\infty$.
Such a metric is unique up to coarse equivalence. 
We call this metric space the \emph{coarse disjoint union} of $\{X^{(m)}\}_m$ 
and denote it by $\coprod_m X^{(m)}$.
However, it is more convenient to allow our metric $d$ to take the value $\infty$, 
and define the distance between two points of distinct components to be $\infty$. 
We call this (generalized) metric space the \emph{disjoint union} of $\{X^{(m)}\}_m$ 
and denote it simply by $\bigsqcup_m X^{(m)}$, as we deal with it most of the time 
throughout this paper. These two notions of a disjoint union do not make much 
difference and the precise relation between these will be described in the end 
of Section~\ref{section:reptheory}.

\subsection{Algebraic uniform Roe algebra and group algebra}

For a uniformly locally finite metric space $(X, d)$,
the algebraic uniform Roe algebra $\mathbb{C}_\mathrm{u} [X]$ is defined to be 
the collection of all matrices indexed by $X$ whose propagation is finite. More precisely,
\[
\mathbb{C}_\mathrm{u} [X]
 = \left\{[a_{x, y}]_{x, y \in X} \Bigm| \sup_{x, y} |a_{x, y}| < \infty
\mbox{ and }\mathrm{prop}(a)<\infty\right\}, 
%\sup \{d(x, y) \ |\ a_{x, y} \neq 0\} < \infty \right\}.
\]
where $\mathrm{prop}(a)=\sup \{d(x, y) \ |\ a_{x, y} \neq 0\}$ is the \emph{propagation} 
of a matrix $a = [a_{x, y}]_{x, y \in X}$.
Usual multiplication between two matrices defines the product on $\mathbb{C}_\mathrm{u} [X]$.
The conjugate transpose $a \mapsto a^* = \left[ \overline{a_{y, x}} \right]_{x, y \in X}$ defines the involution $*$ on the algebra.
The diagonal algebra of $\mathbb{C}_\mathrm{u} [X]$ is canonically isomorphic to 
the algebra $\ell_\infty(X)$ of bounded functions and hence simply denoted by $\ell_\infty(X)$. 

Recall that $t=[t_{x,y}]_{x,y}\in{\mathbb C}_{\mathrm u}[X]$ is called a \emph{partial translation} 
if there is a bijection $\phi_t$ from a subset $\mathop{\mathrm{dom}}(\phi_t)\subset X$ 
onto a subset $\mathop{\mathrm{ran}}(\phi_t)\subset X$ such that 
$t_{x,y}=1$ if $y\in A$ and $x=\phi_t(y)$ else $t_{x,y}=0$. 
We will identify the partial translation $t$ with the partially-defined bijection $\phi_t$. 
A partial translation which is a bijection on $X$ is called a \emph{full translation}. 
We record the following well-known fact as a lemma.

\begin{lemma}\label{lemma:algroestr}
If $\mathrm{Cay}(G)$ is the Cayley metric space of a finitely generated group $G$, then 
${\mathbb C}_{\mathrm u}[\mathrm{Cay}(G)]$ is isomorphic to the algebraic crossed product 
$\ell_\infty(G)\rtimes G$. 
In general, if $X$ is a uniformly locally finite metric space and $\Gamma_X$ denotes 
the group of full translations in ${\mathbb C}_{\mathrm u}[X]$,
then ${\mathbb C}_{\mathrm u}[X]=\mathrm{span}(\ell_\infty(X)\Gamma_X)$. 
\end{lemma}
\begin{proof}
We sketch the proof of the second assertion for the reader's convenience. 
It suffices to show that every partial translation $t$ belongs to the latter set. 
For the bijection $t\colon A\to B$ as above, there is a 
partition $A=\bigsqcup_{i=0}^2 A_i$ such that $t|_{A_0}=\mathrm{id}_{A_0}$ 
and $t(A_i)\cap A_i=\emptyset$ for $i=1,2$ (take a maximal $A_1$ as such). 
Then, $(t|_{A_i})\sqcup(t^{-1}|_{t(A_i)})$ extends 
to a full translation by setting it identity off $A_i\sqcup t(A_i)$.
\end{proof}

The maximal C$^*$-norm on $\mathbb{C}_\mathrm{u} [X]$ is defined as follows:
\[\| a \|_\mathrm{max} = \sup \{ \| \pi(a) \| \mid \pi \colon \mathbb{C}_\mathrm{u} [X] \to \mathbb{B}(\mathcal{H}) *\textrm{-representation on a Hilbert space}\}. \]
Denote by $\mathrm{C}^*_\mathrm{u, max} [X]$ the completion of $\mathbb{C}_\mathrm{u} [X]$ 
with respect to $\| \cdot \|_\mathrm{max}$, and call it the \textit{maximal uniform Roe algebra}.
Note that when $X$ has property $\mathrm{A}$, the norm $\| \cdot \|_\mathrm{max}$ coincides with 
the norm $\| \cdot \|_{{\mathbb B}(\ell_2(X))}$ as an operator on $\ell_2(X)$ 
(see Proposition 1.3 in \cite{SpakulaWillettCrelle}).

For a group $G$, we denote by $\mathrm{C}^*_\mathrm{max} [G]$ 
the completion of the group algebra $\mathbb{C} [G]$ with respect to 
the maximal C$^*$-norm $\| \cdot \|_\mathrm{max}$. 
This $\mathrm{C}^*$-algebra is called the \textit{maximal} (a.k.a.\ \textit{full}) 
\textit{group $\mathrm{C}^*$-algebra}.
We have an natural embedding of the group algebra
$\mathbb{C} [G]$ into $\mathbb{C}_\mathrm{u} [\mathrm{Cay}(G)]$, given by $\xi \mapsto [\xi(g h^{-1})]_{g, h}$,
but the corresponding homomorphism $\mathrm{C}^*_\mathrm{max} [G]\to\mathrm{C}^*_\mathrm{u, max} [X]$ is not 
faithful unless $G$ is amenable. 

\subsection{\bf Sum of squares in $*$-algebras}
For a $*$-algebra, a notion of positivity is defined as follows.

\begin{definition}
For a $*$-algebra $\mathcal{A}$, the cone of \emph{sums of squares} 
is defined to be
\[
\Sigma^2\mathcal{A} = \{ \sum_{i=1}^n \xi_i^* \xi_i : n\in{\mathbb N},\,\xi_1,\ldots,\xi_n\in \mathcal{A}\}.
\]
In case $\mathcal{A} = \mathbb{C}_\mathrm{u} [\mathrm{Cay}(G)]$ or $\mathcal{A} = \mathbb{C} [G]$, 
we say an element $x\in \mathcal{A}$ is a \emph{sum of squares of $n$ elements with propagation at most} $R$ 
if there are $n$ elements $\{\xi_i \}_{i = 1}^{n}$ in $\mathcal{A}$ such that 
$\mathrm{prop}(\xi_i)\le R$ and $x=\sum_{i =1}^n \xi_i^* \xi_i$.
We denote by $\Sigma^2_{n,R}\mathcal{A}$ the set of sums of squares of $n$ elements with propagation at most $R$.
\end{definition}

\begin{lemma}\label{lemma: square sum}
Let $G$ be an amenable $k$-marked group.
Let $x$ be an element in the group algebra $\mathbb{C} [G]$.
Suppose that $x = \sum_{i =1}^n \xi_i^* \xi_i$  is a sum of squares 
of $n$ elements $\{\xi_i\}_{i = 1}^{n}$ in $\mathbb{C}_{\mathrm{u}} [\mathrm{Cay}(G)]$ with propagation at most $R$.
Then $x$ is also a sum of squares as an element of 
$\mathbb{C} [G]$.
More precisely, 
$x$ is a sum of squares of $(n \times \sharp(B(1_G, R)))$ elements 
in $\mathbb{C} [G]$ with propagation at most $R$.
\end{lemma}
\begin{proof}
Fix a mean $\Psi$ on $\ell_\infty(G)$ which is invariant under the left translation 
action $\left\{l_g \colon \ell_\infty(G) \to \ell_\infty(G) \right\}_{g \in G}$.
Note that every element $x$ in $\mathbb{C}_{\mathrm{u}} [\mathrm{Cay}(G)]$ is uniquely 
written as $x=\sum_g \xi_g g$, where $\xi_g \in \ell_\infty(G) \subset \mathbb{C}_\mathrm{u} [\mathrm{Cay}(G)]$ 
(all but finitely many are zero) and $g \in G \subset \mathbb{C}[G]$. 
We extend the mean $\Psi$ to 
$\tilde{\Psi}\colon \mathbb{C}_\mathrm{u} [\mathrm{Cay}(G)] \to \mathbb{C}[G]$
by $\tilde{\Psi}(\sum_g \xi_g g) = \sum_g \Psi(\xi_g) g$.
Since $x = \sum_{i =1}^n \xi_i^* \xi_i \in \mathbb{C}[G]\cap\Sigma^2\mathbb{C}_\mathrm{u} [\mathrm{Cay}(G)]$ 
implies $x = \sum_{i =1}^n\tilde{\Psi}(\xi_i^* \xi_i)$, 
it suffices to show $\tilde{\Psi}(\xi^*\xi) \in \Sigma^2_{N,R} \mathbb{C}[G]$ 
for every $\xi\in\mathbb{C}_\mathrm{u} [\mathrm{Cay}(G)]$ such that $\mathrm{prop}(\xi)\le R$.
Here $N=\sharp(B(1_G, R))$. 

Let $\xi=\sum_{g \in B(1_G, R)} \xi_g g$ be given. Then, 
\[
\tilde{\Psi}(\xi^*\xi) 
 = \sum_{g, h} \tilde{\Psi}\left( l_{h^{-1}} \left(\overline{\xi_h}{\xi_g}\right) h^{-1}g \right)
 = \sum_{g, h} \Psi( \overline{\xi_h}{\xi_g} ) h^{-1}g.
\]
Since $\Psi \colon \ell_\infty(G) \to \mathbb{C}$
is a positive linear functional, 
the matrix $[\Psi \left(\overline{\xi_h}{\xi_{g}} \right)]_{g,h \in B(1_G, R)}$ is positive semi-definite. 
Considering the square root of the matrix,
we obtain vectors $\{\alpha_g\}_{g \in B(1_G, R)} \subset \mathbb{C}^N$ satisfying that
$\Psi \left(\overline{\xi_h}{\xi_{g}} \right) = \langle \alpha_g, \alpha_h \rangle$.
Write $\alpha_g=(\alpha_g^{(i)})_{i=1}^N \in \mathbb{C}^N$. 
Then, one sees
\[
\tilde{\Psi}(\xi^*\xi) 
 = \sum_{g, h} \langle \alpha_g, \alpha_h \rangle h^{-1} g
% = \sum_{g, h} \sum_{i = 1}^N \overline{\alpha_h^{(i)}} \alpha_g^{(i)} \, h^{-1} g
 =  \sum_{i = 1}^N \left( \sum_h \alpha_h^{(i)}  h\right)^* \left( \sum_g \alpha_g^{(i)}  g\right)
\]
and so $\tilde{\Psi}(\xi^*\xi)\in\Sigma^2_{N,R}\mathbb{C}[G]$. 
This completes the proof. 
(In effect, we have shown that $\tilde{\Psi}$ is completely positive in the sense of Section 12 in \cite{OzawaCEC}.)
\end{proof}

\subsection{Positive elements in maximal algebras}

As a simple application of theory of semi-pre-$\mathrm{C}^*$-alge\-bras, 
we obtain the following.
For details, see Proposition $15$ in \cite{Schmudgen} or Theorem $1$ in \cite{OzawaCEC}.

\begin{proposition}\label{proposition: positive elements in algebra}
Let ${\mathcal A}$ be either the group algebra $\mathbb{C} [G]$ of a group $G$ or 
the algebraic uniform Roe algebra $\mathbb{C}_\mathrm{u}[X]$ of 
a uniformly locally finite metric space $X$.
Let $a\in{\mathcal A}$ be a self-adjoint element.
Then the element $a$ is positive in respectively 
$\mathrm{C}^*_\mathrm{max} [G]$ or $\mathrm{C}^*_\mathrm{u, max}[X]$  
if and only if $a + \epsilon 1 \in \Sigma^2{\mathcal A}$ for every $\epsilon>0$.
\end{proposition}

\subsection{\bf On property (T)}

For a $k$-marked group $(G, s_1, \ldots, s_k)$,
let $\Delta$ denote the (nonnormalized) Laplacian
\[2k - \sum_{j = 1}^{k} \left( s_j + s_j^{-1} \right) \in \mathbb{C} [G].\]
Property (T) of Kazhdan can be formulated as follows.

\begin{proposition}[See Lemma 12.1.8 in \cite{OzawaBook}]\label{proposition: (T)}
The following are equivalent:
\begin{itemize}
\item
The group $G$ has property $\mathrm{(T)}$,
\item
The spectrum of $\Delta$ in the C$^*$-algebra $\mathrm{C}^*_\mathrm{max} [G]$ has gap. More precisely,
there exists $\nu > 0$ such that 
the spectrum of $\Delta$ is included in $\{0\} \cup [\nu, \infty)$.
\end{itemize}
\end{proposition}

\subsection{Geometric property (T)}
Let $\{ X^{(m)} \}_m$ be a sequence of finite connected graphs whose degree is uniformly bounded.
A notion called \textit{geometric property $\mathrm{(T)}$} is defined 
for the disjoint union $X = \bigsqcup_{m =1}^\infty X^{(m)}$, or more generally for (weakly) 
monogenic coarse spaces having bounded geometry in \cite{WillettYu13}.
If $X$ has geometric property $\mathrm{(T)}$ and the cardinality $\sharp\left( X^{(m)} \right)$ 
of components tends to $\infty$, then it is a sequence of expander graphs. 
The converse need not hold, see the introduction.
Geometric property $\mathrm{(T)}$ is originally introduced by Willett--Yu in their study of 
the maximal coarse Baum--Connes conjecture \cite{WillettYuII}.

In this paper, we take the second condition in the following proposition as a definition 
of geometric property $\mathrm{(T)}$.
Let $\Delta_m$ be the (nonnormalized) discrete Laplacian on $\ell_2(X^{(m)})$:
\begin{align*}
(\Delta_m)_{x, y}:=\left\{ \begin{array}{ll}
-1 & {\rm if\ } d(x, y)=1 \\
\mathrm{deg}(x) & {\rm if\ } x=y\\
0 & {\rm\ otherwise}\\
\end{array}.\right.
\end{align*}
The sequence $\Delta = (\Delta_m)_{m \in \mathbb{N}}$ is an element of $\mathbb{C}_\mathrm{u} [X]$ whose propagation is $1$.
It is called \emph{Laplacian}. The Laplacian $\Delta \in \mathbb{C}_\mathrm{u} [X]$ is also defined 
for a general uniformly locally finite metric space $X$. See Section 5 in \cite{WillettYu13}.

\begin{proposition}[See Proposition 5.7 in \cite{WillettYu13}]\label{lemma: spectral gap}
The following are equivalent:
\begin{enumerate}
\item
The space $X$ has geometric property $\mathrm{(T)}$.
\item
The Laplacian $\Delta$ in the maximal uniform Roe algebra $\mathrm{C}^*_\mathrm{u, max} [X]$ has spectral gap.
More precisely, there exists a positive number $\nu$
such that the spectrum of $\Delta$ in $\mathrm{C}^*_\mathrm{u, max} [X]$
is included in $\{0\} \cup [\nu, \infty)$.
\end{enumerate}
\end{proposition}

In this paper, we will study when the disjoint union $X=\bigsqcup_{m\in{\mathbb N}} X^{(m)}$ 
of a sequence $\{X^{(m)}\}$ of (finite Cayley) metric spaces has geometric property $(\mathrm{T})$.
The answer is quite simple when $X$ is a \emph{finite} disjoint union. 

\begin{corollary}\label{corollary:finitedisjointunion}
The disjoint union $X=\bigsqcup_{i=1}^n X^{(i)}$ of finitely many spaces $\{ X_i \}_{i=1}^n$ 
has geometric property $(\mathrm{T})$ if and only if each $X^{(i)}$ 
has geometric property $(\mathrm{T})$. 
\end{corollary}

\begin{proof}
Let $\Delta_i$ denote the Laplacian for $X^{(i)}$. 
It is not difficult to see that $\Delta=(\Delta_i)_{i=1}^n$ under 
the canonical isomorphism 
$\mathrm{C}^*_\mathrm{u, max}[X]=\bigoplus_{i=1}^n\mathrm{C}^*_\mathrm{u, max}\left[X^{(i)}\right]$.
\end{proof}

\section{From geometric property (T) to property (T)}\label{sec:geomTtoT}

Let us start the proof of Theorem \ref{theorem:main}: $(\ref{cond:maingeot})\Rightarrow(\ref{cond:maincaybdy})$;
assuming that the disjoint union 
$\bigsqcup_{m \in \mathbb{N}} \mathrm{Cay} \left(G^{(m)}, s_1^{(m)}, \ldots, s_k^{(m)} \right)$ 
has geometric property $\mathrm{(T)}$,
we shall prove that every member   
$\left( G, s_1, \ldots, s_k\right)$ of the Cayley boundary 
$\overline{\left\{ G^{(m)} \right\}}^{\mathrm{Cayley}} \setminus \left\{ G^{(m)} \right\}$ 
has property $\mathrm{(T)}$ of Kazhdan.
Replacing with a subsequence, we may assume that the sequence
$G^{(m)}$ converges to $G$ with respect to the Cayley topology.
See also Corollary~\ref{corollary:finitedisjointunion}.

Let $\Delta_m$ denote the discrete (nonnormalized) Laplacian in the group algebra $\mathbb{C} \left[ G^{(m)} \right]$,
which is given by
\[
\Delta_m = 2k - \sum_{j = 1}^{k} \left(s_j^{(m)} + {s_j^{(m)}}^{-1} \right).
\]
We view $\Delta_m$ also as an element in $\mathbb{C}_{\mathrm{u}} [\mathrm{Cay} \left( G^{(m)} \right)]$ 
and consider the direct product $\Delta = \left( \Delta_m \right) _m
\in \mathbb{C}_\mathrm{u} \left[ \bigsqcup_{m} \mathrm{Cay} \left( G^{(m)} \right) \right]$.
By assumption, the spectrum of $\Delta$ in the maximal uniform Roe algebra 
$\mathrm{C}^*_\mathrm{u, max} \left[ \bigsqcup_{m} \mathrm{Cay} \left( G^{(m)} \right) \right]$ 
is included in $\{0\} \cup [\nu, \infty)$ for some $\nu>0$
(Proposition \ref{lemma: spectral gap}).
Thus the element 
$\Delta^2 - \nu \Delta \in \mathbb{C} _{\mathrm{u}} \left[ \bigsqcup_{m} \mathrm{Cay} \left( G^{(m)} \right) \right]$ 
is positive in the $\mathrm{C}^*$-alge\-bra 
$\mathrm{C}^*_\mathrm{u, max} [ \bigsqcup_{m} \mathrm{Cay} \left( G^{(m)} \right)]$, 
by the spectral mapping theorem. 
By Proposition \ref{proposition: positive elements in algebra}, for every positive number $\epsilon$, 
one has $\Delta^2 - \nu \Delta + \epsilon = \sum_{i=1}^n \eta_i^* \eta_i 
\in \Sigma^2\mathbb{C} _{\mathrm{u}} \left[ \bigsqcup_{m} \mathrm{Cay} \left( G^{(m)} \right) \right]$.
Let $R = \max_i \mathrm{prop}(\eta_i)$ and choose a large natural number $m$ such that 
there exists a partial isomorphism $\phi \colon B \left( 1_{G^{(m)}}, 2R \right) \to B \left( 1_{G}, 2R \right)$ between balls of $G^{(m)}$ and $G$.
Here, $\phi$ satisfies conditions in Lemma
\ref{lemma: partial isomorphism}.
Taking $m$-th entry of $\Delta$ and $\eta_i$,
we see that ${\Delta_m}^2 - \nu \Delta_m + \epsilon \in \Sigma^2_{n,R}\mathbb{C}_{\mathrm{u}} \left[\mathrm{Cay}\left(G^{(m)}\right) \right]$. 
By Lemma~\ref{lemma: square sum},
${\Delta_m}^2 - \nu \Delta_m + \epsilon \in \mathbb{C}  \left[ G^{(m)} \right]$ is a sum of squares of operators  $\xi_i$ in the group algebra $\mathbb{C} \left[ G^{(m)} \right]$ with propagation at most $R$. 
Extending the partial isomorphism $\phi$ to the linear map
$\phi \colon \mathbb{C}  \left[ B \left(1_{G^{(m)}}, 2R\right) \right] \to \mathbb{C} \left[ B \left(1_{G}, 2R \right) \right]$
and applying it to
\[ {\Delta_m}^2 - \nu \Delta_m + \epsilon = 
\sum_i \xi_i^* \xi_i,\]
we see that the Laplacian $\Delta_G$ in the group algebra $\mathbb{C} [G]$ satisfies
\[{\Delta_G}^2 - \nu \Delta_G + \epsilon = 
\sum_i \phi(\xi_i)^* \phi(\xi_i).\]
This implies that ${\Delta_G}^2 - \nu \Delta_G + \epsilon$ is positive 
in the maximal group algebra $\mathrm{C}^*_\mathrm{max} [G]$.
Since $\epsilon>0$ was arbitrary, the spectrum of $\Delta_G$
is included in $\{0\} \cup [\nu, \infty)$, by the spectral mapping theorem.
This means that $G$ has property $\mathrm{(T)}$.
\qed

The key step in the above proof is to show that 
$\Delta_G^2 - \nu \Delta_G + \epsilon$ is a sum of squares for every $\epsilon > 0$.
To check whether a finitely generated group $G$ has property $\mathrm{(T)}$, 
we in fact do not need the extra $\epsilon$. 

\begin{theorem}[\cite{OzawaSquareSum}]\label{theorem:noepsilon}
A marked group $G$ has property $\mathrm{(T)}$ if and only if there exist a positive number $\nu$ and $\xi_1, \ldots, \xi_n \in \mathbb{R} [G]$ such that
\[\Delta_G^2 - \nu \Delta_G = \sum_{i =1}^n \xi_i^* \xi_i.\]
\end{theorem}

\section{From property (T) to geometric property (T)}\label{Section:TtogeomT}
Shalom (Theorem 6.7 in \cite{ShalomUniformAction}) has showed that every property $(\mathrm{T})$ group is 
a quotient of finitely presented property $(\mathrm{T})$ group, and more generally that 
property $(\mathrm{T})$ is an \textit{open} property in the space of marked groups, 
namely, the set of all $k$-marked property $(\mathrm{T})$ groups is open in $\mathcal{G}(k)$.
(This fact also follows from Theorem~\ref{theorem:noepsilon} above.) 
See a survey of Y. Stalder \cite{Stalder} for a more general result. 
Recall that ${\mathcal Q}_\Gamma$ stands for the set of marked group quotients of $\Gamma$.

\begin{proposition}\label{proposition:(T) is a quotient of RF (T)}
Every $k$-marked group $G$ with property $\mathrm{(T)}$ is a quotient of 
a finitely presented $k$-marked group $\Gamma$ with property $\mathrm{(T)}$. 
Moreover, for every sequence $\{G^{(m)}\}_m$ in $\mathcal{G}(k)$ 
that converges to $G$ and for every such finitely presented $\Gamma$,
there exists $N\in\mathbb{N}$ such that $\{ G^{(m)} : m\geq N\} \subset {\mathcal Q}_\Gamma$.
\end{proposition}

\begin{proof}
The first assertion is proved in \cite{ShalomUniformAction} and the second follows 
from Lemma~\ref{lemma: quotients form clopen subset}. See also Lemma~\ref{lemma:convergence from the above}.
\end{proof}

\begin{lemma}\label{lemma:finitely many RF (T)}
Let $K$ be a nonempty closed subset of $\mathcal{G}(k)$.
If every element of $K$ has property $\mathrm{(T)}$,
then there exist finitely many finitely presented property $\mathrm{(T)}$ groups 
$\Gamma_1, \ldots, \Gamma_n \in \mathcal{G}(k)$ such that
$K\subset\bigcup_i {\mathcal Q}_{\Gamma_i}$.
\end{lemma}

\begin{proof}
By Proposition~\ref{proposition:(T) is a quotient of RF (T)}, $\{{\mathcal Q}_\Gamma \ | \ \Gamma \textrm{\ is a f.p.\ group with property $\mathrm{(T)}$}\}$ is an \textit{open} covering of $K$. 
Since the subset $K$ is compact,
there exist finitely many groups $\Gamma_1, \ldots, \Gamma_n$ such that
$K \subset \bigcup_i {\mathcal Q}_{\Gamma_i}$.
\end{proof}

We begin the proof of Theorem \ref{theorem:main}: $(\ref{cond:maincaybdy})\Rightarrow(\ref{cond:maingeot})$.
Suppose that every $k$-marked group
$G$ in the boundary
$\overline{\left\{ G^{(m)} \right\}}^{\mathrm{Cayley}} \setminus \left\{ G^{(m)} \right\}$ has property $\mathrm{(T)}$.
Then every group in the Cayley closure
$\overline{\left\{ G^{(m)} \right\}}^{\mathrm{Cayley}}$ has property $\mathrm{(T)}$.
By Lemma \ref{lemma:finitely many RF (T)}, there exists a decomposition of indices 
$\mathbb{N} = I_1 \sqcup I_2 \sqcup \cdots \sqcup I_n$ 
such that $G^{(m)}$ is a quotient of $\Gamma_i$ for every $m \in I_i$. 
By Corollary~\ref{corollary:finitedisjointunion}, it suffices to show that 
each $X^{(i)}=\bigsqcup_{m \in I_i} \mathrm{Cay} \left(G^{(m)}\right)$ has 
geometric property $\mathrm{(T)}$.
The group homomorphisms $\phi^{(m)} \colon \Gamma_i \to G^{(m)}$, $m \in I_i$, 
induce the $*$-homo\-mor\-phism
\[ \phi_i \colon \mathbb{C} [\Gamma_i] 
\longrightarrow
\prod_{m \in I_i} \mathbb{C} \left[ G^{(m)} \right]
\subset
\mathrm{C}^*_{\mathrm{u, max}}\left[ X^{(i)} \right]
\]
which sends the Laplacian $\Delta_i \in \mathbb{C} [\Gamma_i]$ of 
the group $\Gamma_i$ to the Laplacian of $X^{(i)}$. 
Thus Proposition~\ref{lemma: spectral gap} implies 
that $X^{(i)}$ has geometric property $\mathrm{(T)}$.
\qed

\section{Remarks}\label{section:remarks}
\subsection*{Uniformity of spectral gaps}
Let $\nu(G, S)$ be the infimum of strictly positive spectra
of the discrete Laplacian $\Delta_G$ in $\mathrm{C}^*_\mathrm{max} [G]$.

\begin{proposition}\label{Proposition: Uniform Kazhdan}
Let $K$ be a nonempty closed subset of $\mathcal{G}(k)$ consisting 
of groups with Kazhdan's property $\mathrm{(T)}$.
Then there exists a strictly positive number $\nu$ such that
$\nu \le \nu(G, S)$ for every $(G, S) \in K$.
\end{proposition}

\begin{proof}
By Lemma \ref{lemma:finitely many RF (T)}, there exist finitely many groups $\Gamma_1, \Gamma_2, \ldots, \Gamma_n$
with property $\mathrm{(T)}$ such that every member $G$ of $K$
is a quotient of one of $\Gamma_i$.
Since every representation of $G$ provides a representation of $\Gamma_i$, we have
$ 0<\min_i \nu(\Gamma_i) \le \nu(G)$.
\end{proof}

Proposition~\ref{Proposition: Uniform Kazhdan} roughly states that for a (nonempty) compact set $K$ in $\mathcal{G}(k)$, uniformity concerning property $(\mathrm{T})$ is automatic once all of the members of $K$ have property $(\mathrm{T})$. In \cite{MimuraSakoI}, two of the authors have showed that concerning amenability, a similar phenomenon occurs, see the proof Theorem~5.1 there. On the other hand, two of the authors have revealed that \textit{concerning a-$\mathrm{T}$-menability, the uniformity is not automatically guaranteed}. For details, see the part II paper \cite{MimuraSakoII}.
In \cite{MimuraSakoIV},
we further study the spectral gap and Kazhdan constant.
These constants define functions on the space of marked groups.
We will prove that these are lower semi-continuous functions on $\mathcal{G}(k)$.  (Note that for spectral gaps, it can be also derived 
from the main result of \cite{OzawaSquareSum}. See Theorem~\ref{theorem:noepsilon}.) 
We also study spectral gaps and Kazhdan-type constants for isometric actions on certain metric spaces, including uniformly convex Banach spaces and complete $\mathrm{CAT}(0)$ spaces.

\subsection*{Uniformity on amenability}

\begin{proposition}
Let $\left\{ \left( G^{(m)}, s^{(m)}_1, \ldots, s^{(m)}_k \right) \right\}_{m \in \mathbb{N}}$ be a sequence of amenable $k$-marked groups. 
Suppose that the sequence converges to a $k$-marked group $G^{(\infty)}$.
Then the following conditions are equivalent:
\begin{enumerate}
\item\label{item:property A}
the metric space $\bigsqcup_m \mathrm{Cay} \left( G^{(m)} \right)$ has property $\mathrm{A}$ of Yu,
\item\label{item:amenable group}
the limit group $G^{(\infty)}$ is amenable,
\item\label{item:quotient of ame}
there exists an amenable $k$-marked group $\Gamma$ such that for every $m$, the group $G^{(m)}$ is a quotient of $\Gamma$.
\end{enumerate}
\end{proposition}

\begin{proof}
By Theorem 5.1 in \cite{MimuraSakoI}, conditions $(\ref{item:property A})$ and $(\ref{item:amenable group})$ are equivalent.
Suppose condition (\ref{item:quotient of ame}) holds.
The set ${\mathcal Q}_\Gamma$ forms a closed subset in the space of $k$-marked groups and consists of amenable groups.
It follows that the closure $\overline{\{G^{(m)}\}_m}^\mathrm{Cay}$ is contained in ${\mathcal Q}_\Gamma$. 
In particular, condition (\ref{item:amenable group}) holds.
Conversely, suppose condition (\ref{item:amenable group}) holds.
Define a $k$-marked group $(\Gamma, \gamma_1, \ldots, \gamma_k)$ by the subgroup of 
$\prod_{m=1}^\infty G^{(m)}$ generated by $\gamma_1 = \left( s^{(m)}_1 \right)_m$, $\ldots$, $\gamma_k = \left( s^{(m)}_k \right)_m$.
Each $G^{(m)}$ is a quotient of $\Gamma$.
Note that the limit group $G^{(\infty)}$ is also a quotient of $\Gamma$.
Let $H$ be the kernel of the quotient map $\Gamma \to G^{(\infty)}$.
The group $H$ is equal to
\[\left\{ \left( g^{(m)} \right)_m \in \Gamma \mid
\exists M \in \mathbb{N}, \forall m \ge M, g^{(m)} = 1_{G^{(m)}} 
\right\}.\]
The subgroup $H_M$ defined by
\[\left\{ \left( g^{(m)} \right)_m \in \Gamma \mid 
\forall m \ge M, g^{(m)} = 1_{G^{(m)}} 
\right\}\]
is amenable, since $\bigoplus_{m =1}^M G^{(m)}$ is amenable.
Since $H$ is an increasing union of amenable groups $H_M$,
it is amenable.
Since $\Gamma$ is in the middle of the short exact sequence
\[1 \to H \to \Gamma \to G^{(\infty)} \to 1,\]
$\Gamma$ is also amenable.
Therefore condition (\ref{item:quotient of ame}) holds.
\end{proof}

\section{Cohomological property (T)}\label{section:cohomologicalT}
By the well-known Delorme--Guichardet theorem (Theorem 2.12.4 in \cite{BookBekkadelaHarpeValette}),
property (T) for a (locally compact $\sigma$-compact) group $G$ can be characterized 
by vanishing of the first cohomology group $H^1(G,{\mathcal H})$ of 
every unitary $G$-module ${\mathcal H}$. 
In this section, we study this phenomenon for coarse spaces. 
Namely, we will introduce cohomological property (T) 
for uniformly locally finite metric spaces and prove that 
it implies geometric property (T). 
For this purpose, we develop a cohomology theory for such spaces, 
in analogy with the cohomology theory for group actions. 
See \cite{Elek},  Chapter 8 in \cite{GromovAsym}, \cite{oguni}, \cite{Pansu}, 
and Chapter 5 in \cite{RoeLectureNote} for relevant results. 
We first work purely algebraically. 
So, we consider a unital algebra ${\mathcal A}$ (over ${\mathbb C}$)
together with an ``augmentation'' map $\omega$ from ${\mathcal A}$ onto a unital subalgebra ${\mathcal D}\subset{\mathcal A}$
that satisfies $\omega|_{{\mathcal D}}=\mathrm{id}_{{\mathcal D}}$ and $\omega(ab)=\omega(a\omega(b))$ for $a,b\in{\mathcal A}$.
(Although we are content with the unital setting, one may also want to look at a non-unital algebra ${\mathcal A}$ 
and a unital subalgebra ${\mathcal D}$ in the multiplier of ${\mathcal A}$.)
It follows that ${\mathcal L}:=\ker\omega$ is a left ideal of ${\mathcal A}$.
Prototypical examples are group algebras ${\mathbb C}[G]$ and
the unit character $\omega\colon{\mathbb C}[G]\to{\mathbb C}\subset{\mathbb C}[G]$;
and the algebraic uniform Roe algebras ${\mathbb C}_{\mathrm u}[X]$ of a uniformly locally finite metric 
space $X$ and $\omega\colon{\mathbb C}_{\mathrm u}[X]\to\ell_\infty(X)\subset{\mathbb C}_{\mathrm u}[X]$, 
given by $\omega(a)(x)=\sum_y a_{x,y}$ for $a=[a_{x,y}]_{x,y\in X}\in{\mathbb C}_{\mathrm u}[X]$ 
(this augmentation map $\omega$ is denoted by $\Phi$ in \cite{WillettYu13}). 

For a given left ${\mathcal A}$-module ${\mathcal M}$, we define the cohomology groups $H^n({\mathcal A},{\mathcal M})$ to be
the relative Ext-groups of the ${\mathcal D}$-alge\-bra ${\mathcal A}$ with coefficients in ${\mathcal M}$ relative to ${\mathcal L}$.
Namely, $H^n({\mathcal A},{\mathcal M})=\ker d_{n+1}/\mathop{\mathrm{ran}} d_n$ for the cochain complex
\begin{align*}
{\mathcal M} &\stackrel{d_1}{\longrightarrow} \mathrm{Hom}_{{\mathcal D}}({\mathcal L},{\mathcal M}) \stackrel{d_2}{\longrightarrow}
 \mathrm{Hom}_{{\mathcal D}}({\mathcal A}\otimes_{{\mathcal D}}{\mathcal L},{\mathcal M}) \stackrel{d_3}{\longrightarrow} \cdots\\
 & \quad\cdots\stackrel{d_n}{\longrightarrow}
     \mathrm{Hom}_{{\mathcal D}}(\overbrace{{\mathcal A}\otimes_{{\mathcal D}}\cdots\otimes_{{\mathcal D}}{\mathcal A}}^{n-1}\otimes_{{\mathcal D}}{\mathcal L},{\mathcal M})
   \stackrel{d_{n+1}}{\longrightarrow}\cdots,
\end{align*}
where $\mathrm{Hom}_{{\mathcal D}}$ is the space of ${\mathcal D}$-module maps and
\begin{align*}
(d_n\theta)(a_1\otimes\cdots\otimes a_n) &=a_1\theta(a_2\otimes\cdots\otimes a_n)\\
&\quad + \sum_{i=1}^{n-1} (-1)^i
  \theta(a_1\otimes \cdots \otimes a_{i-1}\otimes a_ia_{i+1}\otimes a_{i+2}\otimes\cdots\otimes a_n).
\end{align*}
For example, $1$-cocycles are ${\mathcal A}$-module maps $\theta\colon{\mathcal L}\to{\mathcal M}$ and
$1$-coboundaries are those given by ${\mathcal L}\ni a\mapsto av\in{\mathcal M}$ for some $v\in {\mathcal M}$.
In case ${\mathcal A}={\mathbb C}[G]$ and $\omega\colon{\mathbb C}[G]\to{\mathbb C}\subset{\mathbb C}[G]$ is the unit character,
it is not difficult to check that $H^n({\mathcal A},{\mathcal M})$ is nothing but the group cohomology $H^n(G,{\mathcal M})$.
In case ${\mathcal A}={\mathbb C}_{\mathrm u}[X]$, we will denote $H^n({\mathcal A},{\mathcal M})$ by $H^n(X,{\mathcal M})$.

\begin{proposition}\label{proposition:crossedprodcohomology}
Let ${\mathcal D}$ be a unital algebra on which a group $G$ acts 
and ${\mathcal A}={\mathcal D}\rtimes G$ be the algebraic crossed
product, together with the augmentation map $\omega\colon {\mathcal A}\to{\mathcal D}$ given by
\[
\omega(\sum_{g\in G} a_g g ) = \sum_{g\in G} a_g.
\]
Then, for every left ${\mathcal A}$-module ${\mathcal M}$, the homomorphisms
$\iota_*\colon H^n({\mathcal A},{\mathcal M}) \to H^n(G,{\mathcal M})$
induced by the inclusion $\iota\colon{\mathbb C}[G]\hookrightarrow{\mathcal A}$ are isomorphisms.
\end{proposition}
\begin{proof}
We note that $G$ (resp.\ $\{ g-1_G : g\in G\setminus\{1_G\}\}$) is a basis for the 
left ${\mathcal D}$-module ${\mathcal A}$ (resp.\ ${\mathcal L}:=\ker\omega$).
Also note that ${\mathcal L}[G]:=\ker\omega\cap{\mathbb C}[G]=\mathop{\mathrm{span}}_{{\mathbb C}}\{ g-1_G \}$.
It follows that the linear map
\[
{\mathcal D}\otimes_{{\mathbb C}}\overbrace{{\mathbb C}[G]\otimes_{{\mathbb C}}\cdots\otimes_{{\mathbb C}}{\mathbb C}[G]}^{n-1}\otimes_{{\mathbb C}}{\mathcal L}[G]
\to\overbrace{{\mathcal A}\otimes_{{\mathcal D}}\cdots\otimes_{{\mathcal D}}{\mathcal A}}^{n-1}\otimes_{{\mathcal D}}{\mathcal L}
\]
given by 
$a\otimes\xi_1\otimes\cdots\otimes\xi_n\mapsto (a\xi_1)\otimes\xi_2\otimes\cdots\otimes\xi_n$ 
is a ${\mathcal D}$-module isomorphism.
Hence, every $n$-cocycle $\theta$ for ${\mathbb C}[G]$ uniquely extends to an $n$-cocycle for ${\mathcal A}$.
This gives rise to the inverse of $\iota_*$ and so $\iota_*$ is an isomorphism. 
\end{proof}

\begin{corollary}\label{cor:ci1}
Let $G$ be a group and $X=\mathrm{Cay}(G)$. 
Then, for every left ${\mathbb C}_{\mathrm u}[X]$-module ${\mathcal M}$, there
are canonical isomorphisms $H^n(X,{\mathcal M})\cong H^n(G,{\mathcal M})$.
\end{corollary}

\begin{proof}
Apply Proposition~\ref{proposition:crossedprodcohomology} to ${\mathbb C}_{\mathrm u}[X]\cong\ell_\infty(X)\rtimes G$.
\end{proof}

Here we recover a result of Pansu \cite{Pansu} that vanishing of $L^2$-Betti numbers 
is a coarse invariant. 

\begin{corollary}\label{cor:ci2}
Let $G_1$ and $G_2$ be groups which are coarsely equivalent. 
Then, for every $n$ one has $\beta_n(G_1)=0\Leftrightarrow\beta_n(G_2)=0$. 
Here $\beta_n(G)$ denotes the $n$-th $L^2$-Betti number of $G$. 
\end{corollary}
\begin{proof}
If one of $G_1$ and $G_2$ is amenable, then both are amenable (Proposition 3.3.5 in \cite{RoeLectureNote}) 
and have zero $L^2$-Betti numbers by the Cheeger--Gromov theorem (Theorem 6.3.7 in \cite{Lueck}). 
So, we assume that $G_1$ and $G_2$ are not amenable. 
By \cite{Whyte}, there is a bijective coarse equivalence 
between $G_1$ and $G_2$. 
(This may not be true when the groups are amenable \cite{Dymarz}, 
still the algebraic uniform Roe algebras of them are Morita equivalent.) 
Thus, for $X:=\mathrm{Cay}(G_1)=\mathrm{Cay}(G_2)$, 
Corollary \ref{cor:ci1} yields a canonical isomorphism $H^n(G_1,\ell_2(X))\cong H^n(G_2,\ell_2(X))$ 
that preserves the topologies also. Here for a topological $G$-module ${\mathcal M}$, 
the topology on $H^n(G,{\mathcal M})$ is induced from the pointwise convergence topology on 
the space of cochains. 
Since $\beta_n(G)=0$ if and only if $H^n(G,\ell_2(G))=\overline{\{0\}}$, we are done. 
\end{proof}

\begin{definition}
Let $X$ be a uniformly locally finite metric space. 
We say $X$ has \emph{cohomological property} (T) if $H^1(X,{\mathcal H})=0$ for every
$*$-representation of the algebraic uniform Roe algebra ${\mathbb C}_{\mathrm u}[X]$ 
on a Hilbert space ${\mathcal H}$.
\end{definition}

We denote the augmentation left ideal 
$\ker(\omega\colon{\mathbb C}_{\mathrm u}[X]\to\ell_\infty(X))$ 
by ${\mathcal L}_{\mathrm{u}}[X]$. 
Suppose that a $1$-cocycle $\theta\colon {\mathcal L}_{\mathrm{u}}[X]\to{\mathcal H}$ 
is given. Then, $b(t):=\theta(t-\omega(t))$ is a $1$-cocycle for the pseudo-group 
of partial translations $t$ in ${\mathbb C}_{\mathrm u}[X]$. Namely, it satisfies 
the cocycle identity 
$b(st)=b(s|_{ \mathop{\mathrm{ran}}(t)}) + \pi(s|_{ \mathop{\mathrm{ran}}(t)}) b(t)$ 
for $s$ and $t$. 
The following criterion of $1$-coboundaries is handy, as in the case of group $1$-cocycles.

\begin{lemma}\label{lemma:boundedcocycle}
A $1$-cocycle $\theta\colon {\mathcal L}_{\mathrm{u}}[X] \to {\mathcal H}$ is a $1$-coboundary 
if and only if $\sup_t\|\theta(t-\omega(t))\|<\infty$ where the supremum is taken all over 
the partial $($resp.\ full$)$ translations $t$.
\end{lemma}
\begin{proof}
We only have to prove the `if' part. Let $\Gamma_X$ be the group of full translations 
in ${\mathbb C}_{\mathrm u}[X]$.
Since the cocycle $\theta$ is bounded, there is $v\in{\mathcal H}$ such that 
$\theta(t-1)=\pi(t)v-v$ for all $t\in\Gamma_X$, 
by Proposition 2.2.9 in \cite{BookBekkadelaHarpeValette}.
But since ${\mathcal L}_{\mathrm{u}}[X]=\mathrm{span}(\ell_\infty(X)\{ t-1 : t\in\Gamma_X\})$ 
(see Lemma~\ref{lemma:algroestr}),
one has $\theta(a)=\pi(a)v$ for all $a\in{\mathcal L}_{\mathrm{u}}[X]$.
\end{proof}

It is proved in \cite{WillettYu13} that geometric property $(\mathrm{T})$ is a coarse invariant. 
The same is true for the cohomological property $(\mathrm{T})$. 
\begin{theorem}
Let $X$ and $Y$ be uniformly locally finite metric spaces which are coarsely equivalent.
Then, $X$ has cohomological property $(\mathrm{T})$ if and only if $Y$ has it.
\end{theorem}

\begin{proof}
We first prove that if an inclusion $X\hookrightarrow Y$ is a coarse equivalence 
and $X$ has cohomological property $(\mathrm{T})$, then so does $Y$. 
Since $Y$ is uniformly locally finite, there is a finite family $\{ f_i\}$ of  
partial translations such that $\mathop{\mathrm{dom}}(f_i)\subset X$ and 
$\bigcup_i \mathop{\mathrm{ran}}(f_i)=Y$.
It is not too difficult to see that every partial translation on $Y$ is 
a concatenation of partial translations of the form $f_j\circ s_{ij} \circ f_i^{-1}$ 
where $s_{ij}$'s are partial translations on $X$.
Now, suppose that a $1$-cocycle 
$\theta\colon {\mathcal L}_{\mathrm{u}}[Y]\to {\mathcal H}$ is given. 
By cohomological property $(\mathrm{T})$ of $X$, the pseudo-group cocycle 
$b(s):=\theta(s-\omega(s))$ is uniformly bounded on partial translations $s$ of $X$.
By the cocycle identity, $b(f_j\circ s \circ f_i^{-1})$ is also uniformly bounded 
for every $i$ and $j$. It follows that $b$ is bounded on partial translations of $Y$, 
and by Lemma~\ref{lemma:boundedcocycle} it is a $1$-coboundary. 
This proves that $Y$ has cohomological property $(\mathrm{T})$.

In view of Lemma~\ref{lemma:coarseequivstr}, it remains show that 
cohomological property $(\mathrm{T})$ of $X\times\mathbf{n}$ implies that of $X$. 
We identify ${\mathbb C}_{\mathrm u}[\mathbf{n}]$ with ${\mathbb M}_n({\mathbb C})$ and 
${\mathbb C}_{\mathrm u}[X\times\mathbf{n}]$ 
with ${\mathbb M}_n({\mathbb C})\otimes{\mathbb C}_{\mathrm u}[X]$ via 
$[a_{(x,i),(y,j)}]_{(x,i),(y,j)}\leftrightarrow\sum_{i,j} e_{ij}\otimes [a_{(x,i),(y,j)}]_{x,y}$. 
Then, the augmentation map $\omega_{X\times\mathbf{n}}$ for $X\times\mathbf{n}$ is of the form 
$\omega_{X\times\mathbf{n}}=\omega_\mathbf{n}\otimes\omega$ where 
$\omega_\mathbf{n}([a_{i,j}]_{i,j})=[\delta_{i,j}\sum_k a_{i,k}]_{i,j}$ on ${\mathbb M}_n({\mathbb C})$. 
It follows that $\sum_{i,j} e_{ij}\otimes a_{i,j} \in {\mathcal L}_{\mathrm u}[X\times\mathbf{n}]$ 
implies $\sum_j a_{i,j}\in{\mathcal L}_{\mathrm u}[X]$ for every $i$.
Let a $*$-rep\-re\-sen\-ta\-tion 
$\pi\colon {\mathbb C}_{\mathrm u}[X]\to{\mathbb B}({\mathcal H})$ 
and a $1$-cocycle $\theta\colon {\mathcal L}_{\mathrm u}[X]\to{\mathcal H}$ be given. 
Then, we define a $1$-cocycle 
$\theta_n\colon {\mathcal L}_{\mathrm u}[X\times\mathbf{n}]\to{\mathcal H}^{\oplus n}$ 
for the $*$-rep\-re\-sen\-ta\-tion 
$\pi_n=\mathrm{id}\otimes\pi$ of ${\mathbb M}_n({\mathbb C})\otimes{\mathbb C}_{\mathrm u}[X]$ on 
${\mathcal H}^{\oplus n}$ by 
$\theta_n(\sum_{i,j} e_{ij}\otimes a_{i,j})=(\theta(\sum_j a_{i,j}))_{i=1}^n$.
It is indeed a $1$-cocycle and hence there is $\tilde{v}=(v_i)_{i=1}^n\in{\mathcal H}^{\oplus n}$ 
such that $\theta_n(\tilde{a})=\pi_n(\tilde{a})\tilde{v}$ 
for all $\tilde{a}\in {\mathcal L}_{\mathrm u}[X\times\mathbf{n}]$.
It follows that $\theta(a)=\pi(a)v_1$ for all $a\in {\mathcal L}_{\mathrm u}[X]$, 
and $\theta$ is a $1$-coboundary.
\end{proof}

The following finishes the proof of Theorem~\ref{theorem:main}.

\begin{theorem}
For a uniformly locally finite metric space $X$, one has the following. 
\begin{itemize}
\item Cohomological property $(\mathrm{T})$ implies geometric property $(\mathrm{T})$.
\item If the coarse structure of $X$ is induced by a property $(\mathrm{T})$ group $G$ 
$($i.e., $G=\langle S\rangle$ acts on $X$ in such a way that 
$\{ (x,gx) : x\in X,\, g\in S\}$ is a generating controlled subset$)$, 
then $X$ has cohomological property $(\mathrm{T})$.
\item In case $X$ is a disjoint union of the Cayley graphs of finite $k$-marked groups,
$X$ has cohomological property $(\mathrm{T})$ if and only if it has geometric property $(\mathrm{T})$.
\end{itemize}
\end{theorem}
\begin{proof}
Suppose that $X$ does not have geometric property $(\mathrm{T})$.
By Proposition 3.8 in \cite{WillettYu13}, there are $*$-representations $\pi_n$
of ${\mathbb C}_{\mathrm u}[X]$ on ${\mathcal H}_n$ and unit vectors
$\xi_n\in({\mathcal H}_n)_c^\perp$
such that $\|\pi_n(a-\omega(a))\xi_n\|\le4^{-n}\sup_{x,y\in X}
|a_{x,y}|$ for every $a=[a_{x,y}]_{x,y}$ of propagation at most $n$. 
Here $({\mathcal H}_n)_c$ is the space of constant vectors, which consists of the
vectors in ${\mathcal H}_n$ that are annihilated by $\pi_n({\mathcal L}_{\mathrm u}[X])$.
(Remember that our $\omega$ is denoted by $\Phi$ in \cite{WillettYu13}.)
Now, we may define a $1$-cocycle $\theta$ from ${\mathcal L}_{\mathrm u}[X]$
into $\bigoplus_n{\mathcal H}_n$ by $\theta(a)=\sum^{\oplus}_n 2^n \pi_n(a)\xi_n$.
We claim that $\theta$ is not a $1$-coboundary.
For if it were, there would be $\zeta=\sum^{\oplus}_n \zeta_n\in
\bigoplus_n{\mathcal H}_n$ such that
$\theta(a)=\sum^{\oplus}\pi_n(a)\zeta_n$.
This means that $\pi_n(a)\zeta_n=2^n\pi_n(a)\xi_n$ for every $a\in{\mathcal L}_{\mathrm u}[X]$,
or equivalently, $\zeta_n-2^n\xi_n\in({\mathcal H}_n)_c$.
But this implies $P_{({\mathcal H}_n)_c^\perp}\zeta_n=2^n\xi_n$ and contradicts
the fact that $\|\zeta\|<\infty$.
Consequently, $\theta$ is not a $1$-coboundary and $X$ does not have
cohomological property (T). This proves the first assertion.

The second assertion follows from the observation that every $1$-cocycle
$\theta\colon{\mathcal L}_{\mathrm u}[X]\to{\mathcal H}$ is a $1$-coboundary for the
induced unitary
representation of a property (T) group $G$, which fact implies that
$\theta$ is also
a $1$-coboundary for ${\mathbb C}_{\mathrm u}[X]$.
The last assertion follows from 
Theorem~\ref{theorem:main}:$(\ref{cond:maincaybdy})\Leftrightarrow(\ref{cond:maingeot})$,
Lemma~\ref{lemma:finitely many RF (T)}, and the previous two.
\end{proof}

In fact, by adapting the method of \cite{OzawaSquareSum}, one can prove that 
cohomological property $(\mathrm{T})$ implies 
$\Delta^2-\nu\Delta\in\Sigma^2{\mathbb C}_{\mathrm u}[X]$ for some $\nu>0$ (see Theorem~\ref{theorem:noepsilon}). 
The following is a characterization of a Cayley metric space $\mathrm{Cay}(\Gamma)$ to have 
cohomological property $(\mathrm{T})$. Recall that a unitary $G$-rep\-re\-sen\-ta\-tion $\pi$ 
is said to be \emph{weakly regular} if it is weakly contained in the regular 
representation on $\ell_2(G)$. 

\begin{proposition}
Let $X=\mathrm{Cay}(G)$ be the Cayley metric space and assume that 
$X$ has property $\mathrm{A}$ $($or equivalently, $G$ is exact$)$. 
Then, $X$ has cohomological property $(\mathrm{T})$ if and only if 
$H^1(G,{\mathcal H})=0$ for every weakly regular 
unitary $G$-rep\-re\-sen\-ta\-tion ${\mathcal H}$.
\end{proposition}
\begin{proof}
Since $X=\mathrm{Cay}(G)$ has property $\mathrm{A}$, 
every $*$-representation of ${\mathbb C}_{\mathrm u}[X]$ 
is weakly contained in the regular representation on $\ell_2(X)$, 
by Proposition~1.3 in \cite{SpakulaWillettCrelle}. 
Thus the `if' part of the proposition follows from Corollary~\ref{cor:ci1}. 
Conversely, suppose that $H^1(G,{\mathcal H})\neq0$ for 
a weakly regular $G$-rep\-re\-sen\-ta\-tion $\pi$ on ${\mathcal H}$.
We view $\pi$ as a $*$-homo\-mor\-phism from the reduced group $\mathrm{C}^*$-alge\-bra 
$\mathrm{C}^*_{\lambda}(G)$ into 
${\mathbb B}({\mathcal H})$. By Arveson's and Stinespring's theorems applied in tandem to 
$\mathrm{C}^*_{\lambda}(G)\subset\mathrm{C}^*_{\mathrm{u}}[X]$, 
there are a Hilbert space $\hat{\mathcal H}\supset{\mathcal H}$ and a 
$*$-rep\-re\-sen\-ta\-tion 
$\hat{\pi}\colon \mathrm{C}^*_{\mathrm{u}}[X] \to{\mathbb B}(\hat{\mathcal H})$ 
such that $\hat{\pi}(g)|_{\mathcal H}$ coincides with the original $\pi(g)$ for every $g\in G$.
Since $H^1(G,\hat{\mathcal H})\supset H^1(G,{\mathcal H})\neq0$, one has 
$H^1(X,\hat{\mathcal H})\neq0$ by Corollary~\ref{cor:ci1}. 
\end{proof}

It follows that $\mathrm{Cay}(F_2)$ has geometric property $(\mathrm{T})$ 
(Corollary 6.5 in \cite{WillettYu13}), but not cohomological property $(\mathrm{T})$ 
(since $H^1(F_2,\ell_2(F_2))\neq0$). 
Lattices in $\mathrm{SL}(2,{\mathbb C})$ also do not have cohomological property $(\mathrm{T})$ 
(Exemple~3 in \cite{Guichardet}) although they have zero first $L^2$-Betti numbers 
(Theorem~4.1 in \cite{Lueck}).
These examples show that cohomological property $(\mathrm{T})$ is in general strictly stronger than 
geometric property $(\mathrm{T})$. The authors do not know whether these properties are 
equivalent for a disjoint union of finite metric spaces of bounded geometry. 
There are many Cayley metric spaces that have cohomological property $(\mathrm{T})$, 
besides those come from property $(\mathrm{T})$ groups, e.g., $\mathrm{Cay}(F_2\times F_2)$.
Indeed, it is well known that $H^1(G_1\times G_2,{\mathcal H})=0$ for every non-amenable group $G_i$ 
and every weakly regular representation $(\pi,{\mathcal H})$. We sketch the proof of this fact for 
the reader's convenience. Let $b\colon G_1\times G_2\to{\mathcal H}$ be a $1$-cocycle, i.e., 
it satisfies $b(gh)=b(g)+\pi(g)b(h)$ for every $g,h\in G_1\times G_2$. 
Thus, for every $g_1\in G_1$ and $g_2\in G_2$, one has $(\pi(g_2)-1)b(g_1)=(\pi(g_1)-1)b(g_2)$.
Since $G_i$ is not amenable, $\pi$ does not weakly contain the trivial representation, or equivalently,  
there are a finite subset $S_i\subset G_i$ and $C>0$ such that 
$\|v\|\le C\|\sum_{g\in S_i}(\pi(g)-1)v\|$ holds for all $v\in{\mathcal H}$.
This implies that $b$ is bounded on each of $G_i$'s and hence on $G_1\times G_2$. 
Such a $1$-cocycle is a $1$-coboundary (Proposition 2.2.9 in \cite{BookBekkadelaHarpeValette}). 

\section{On the structure of maximal uniform Roe algebras}\label{section:reptheory}

In this section, we develop the representation theory of the 
maximal uniform Roe algebras of a (coarse) disjoint union. 
For the reader's convenience, we recall the basic properties of
$C(K)$-$\mathrm{C}^*$-alge\-bras \cite{Kasparov} in the unital setting.
Let $K$ be a compact Hausdorff space.
A unital $\mathrm{C}^*$-alge\-bra $A$ is called a $C(K)$-$\mathrm{C}^*$-alge\-bra
if it comes together with a $*$-homo\-mor\-phism $\theta$ from $C(K)$ 
into the center of $A$. We will omit writing $\theta$ as if $C(K)$ is 
a subalgebra of $A$.
For each $t\in K$, let $I_t=\overline{C_0(K\setminus\{t\})A}$ be
the corresponding ideal of $A$ (in fact $I_t= C_0(K\setminus\{t\})A$ by Cohen's
factorization theorem), and denote by $\pi_t\colon A\to A/I_t=:A_t$ the
corresponding quotient. Then, every irreducible representation of $A$ factors through some
$\pi_t$, since its restriction to $C(K)$ is a character associated with some
point $t\in K$. It follows that the representation
$A\ni a\mapsto\prod_t\pi_t(a)\in\prod_t A_t$ is faithful.
In particular, $\mathrm{Sp}(a)=\overline{\bigcup_t\mathrm{Sp}(\pi_t(a))}$ for every $a\in A$. 
Each $A_t$ is called a fiber of $A$.
A $*$-homo\-mor\-phism $\sigma\colon A\to B$ between $C(K)$-$\mathrm{C}^*$-alge\-bras
$A$ and $B$ is simply called a morphism if its restriction to $C(K)$ is the identity map.
Such a morphism naturally induces $*$-homo\-mor\-phisms $\sigma_t\colon A_t\to B_t$
on the fibers. Note that a morphism $\sigma\colon A\to B$ is injective if and only if
it is the case for each fiber $\sigma_t\colon A_t\to B_t$.
In particular, a $C(K)$-$\mathrm{C}^*$-alge\-bra $A$ is nuclear if and only if all fibers
$\{ A_t \}_t$ are nuclear.

Now, we consider a $*$-algebra ${\mathcal B}$ containing $C(K)$ in its center. 
Then its universal enveloping $\mathrm{C}^*$-alge\-bra $A=\mathrm{C}^*({\mathcal B})$  
is a $C(K)$-$\mathrm{C}^*$-alge\-bra. For each $t\in K$, the ideal 
${\mathcal J}_t = C_0(K\setminus\{t\}){\mathcal B}$ is dense in the ideal 
$I_t$ of $A$. It is not hard to see that $A_t=A/I_t$ is the universal enveloping 
$\mathrm{C}^*$-alge\-bra of ${\mathcal B}_t:={\mathcal B}/{\mathcal J}_t$.

Let $\left\{\left(G^{(m)},s_1^{(m)},\ldots,s_k^{(m)}\right)\right\}_m$ 
be a sequence of $k$-marked groups and denote by $\sigma_m$ 
the corresponding homomorphism from $F_k=\langle s_1,\ldots,s_k\rangle$ 
onto $G^{(m)}$ that maps $s_i$ to $s_i^{(m)}$. 
Let $X=\bigsqcup_m\mathrm{Cay}\left(G^{(m)},s_1^{(m)},\ldots,s_k^{(m)}\right)$ be 
the disjoint union. 
For $g\in F_k$, let $v_g$ denote the element in ${\mathbb C}_{\mathrm u}[X]$ 
which is represented by the kernel
\[
(v_g)(x,y)=\left\{\begin{array}{cl}
 1 & \mbox{ if $x,y\in G^{(m)}$ and $x=\sigma_m(g)y$ }\\
 0 & \mbox{ otherwise}
\end{array}\right.
\]
and note that
\[
{\mathbb C}_{\mathrm u}[X]=\left\{ \sum_{g\in F_k}\xi_g v_g \Bigm| 
 \xi_g\in\ell_\infty(X)\mbox{ are zero for all but finitely many $g$} \right\}.
\]
The center of the algebraic uniform Roe algebra ${\mathbb C}_{\mathrm u}[X]$ consists of those functions
in $\ell_\infty(X)$ that are constant on each of $G^{(m)}$'s, and so it is 
canonically isomorphic to $\ell_\infty({\mathbb N})$.
We recall that the Gelfand spectrum of $\ell_\infty({\mathbb N})$ is
the Stone--\v Cech compactification $\beta{\mathbb N}$ of ${\mathbb N}$, and $\ell_\infty({\mathbb N})$ is
$*$-iso\-mor\-phic to $C(\beta{\mathbb N})$. Thus, the maximal uniform Roe algebra
$\mathrm{C}^*_{\mathrm{u, max}}[X]$ is a $C(\beta{\mathbb N})$-$\mathrm{C}^*$-alge\-bra.
Let us fix an element $\omega\in\beta{\mathbb N}$ for a while, and
identify it with the corresponding character $\omega\colon\ell_\infty({\mathbb N})\to{\mathbb C}$.
We will denote $\omega(\xi)$ also by $\lim_\omega\xi(m)$. 
We still abuse the notation and denote the corresponding ultrafilter by $\omega$, too. 
Namely, we identify $\omega$ with the family of those subsets $E\subset{\mathbb N}$ 
such that $\omega(1_E)=1$. Here we recall that an ultrafilter is a 
family of non-empty subsets that satisfies the finite intersection property 
and the maximality condition that $E \notin \omega$ implies $({\mathbb N}\setminus E)\in\omega$. 
For example, $n\in{\mathbb N}$ is identified with the principal character
$\ell_\infty({\mathbb N})\ni\xi\mapsto\xi(n)$ and with the principal ultrafilter consisting of
the subsets $E\subset{\mathbb N}$ that contain $n$. 
By the universality of the Stone--\v Cech compactification, 
the map ${\mathbb N}\ni m\mapsto G^{(m)}\in{\mathcal G}(k)$ extends to a continuous map
$\beta{\mathbb N}\ni \omega\mapsto G^{(\omega)}\in{\mathcal G}(k)$.
We note that 
\[
G^{(\omega)}=F_k/\{ w \in F_k : \{ m\in{\mathbb N} \mid \sigma_m(w) = 1 \} \in \omega \}
\]
as a $k$-marked group, with the 
corresponding homomorphism denoted by $\sigma_\omega\colon F_k\to G^{(\omega)}$.
To see the relation between $G^{(\omega)}$ and ${\mathbb C}_{\mathrm u}[X]_\omega$, 
let us observe that
\[
C_0(\beta{\mathbb N}\setminus\{\omega\})\ell_\infty(X)
=\left\{(\xi^{(m)})_{m=1}^\infty\in \prod_{m=1}^\infty\ell_\infty\left(G^{(m)}\right) \biggm| \lim_\omega \|\xi^{(m)}\|=0\right\}
\]
and 
\[
C_0(\beta{\mathbb N}\setminus\{\omega\}){\mathbb C}_{\mathrm u}[X] 
 = \left\{ a=[a_{x,y}]_{x,y}\in{\mathbb C}_{\mathrm u}[X] \Bigm| \lim_\omega\sup_{x,y\in G^{(m)}}|a_{x,y}|=0 \right\}.
\]
Hence, $\pi_\omega(v_g)=1$ in ${\mathbb C}_{\mathrm u}[X]_\omega$ 
$\Leftrightarrow$ $1-v_g \in C_0(\beta{\mathbb N}\setminus\{\omega\}){\mathbb C}_{\mathrm u}[X] $ 
$\Leftrightarrow$ $\sigma_\omega(g)=1$.
This means that $g\mapsto\pi_\omega(v_g)$ gives rise to an 
inclusion $G^{(\omega)}\hookrightarrow{\mathbb C}_{\mathrm u}[X]_\omega$.
Moreover, ${\mathbb C}_{\mathrm u}[X]_\omega$ is canonically isomorphic to 
the algebraic crossed product $\ell_\infty(X)_\omega\rtimes G^{(\omega)}$.
Here $\ell_\infty(X)_\omega
 := \ell_\infty(X)/(C_0(\beta{\mathbb N}\setminus\{\omega\})\ell_\infty(X))$, 
which is nothing but the ultraproduct $\mathrm{C}^*$-alge\-bra of $\ell_\infty\left(G^{(m)}\right)$'s.
We denote by $\xi^{(\omega)}$ the element in $\ell_\infty(X)_\omega$
that corresponds to $\xi=(\xi^{(m)})_{m=1}^\infty\in\prod_{m=1}^\infty \ell_\infty\left(G^{(m)}\right)$.
For $g\in G^{(\omega)}$, let $u_g$ denote the corresponding element in
the maximal (a.k.a.\ full) crossed product $\mathrm{C}^*$-alge\-bra 
$\ell_\infty(X)_\omega\rtimes_{\max}G^{(\omega)}$.

Gathering all the above discussions, we have followed due process to ensure 
canonicality of the canonical maps. 

\begin{theorem}\label{theorem:rep}
Let $X=\bigsqcup_{m\in{\mathbb N}}\mathrm{Cay}\left(G^{(m)},s_1^{(m)},\ldots,s_k^{(m)}\right)$ be as above.
With the above notation, the maximal uniform Roe algebra
$\mathrm{C}^*_{\mathrm{u, max}}[X]$ is a $C(\beta{\mathbb N})$-$\mathrm{C}^*$-alge\-bra and 
for each $\omega\in\beta{\mathbb N}$ there is a canonical $*$-iso\-mor\-phism
\[
\mathrm{C}^*_{\mathrm{u, max}}[X]_\omega \cong \ell_\infty(X)_\omega\rtimes_{\max}G^{(\omega)}
\]
that sends $\sum_{g\in F_k} \xi_g v_g$ in ${\mathbb C}_{\mathrm u}[X]$ to $\sum_{g\in F_k} \xi^{(\omega)}_g u_{\sigma_\omega(g)}$.
\end{theorem}

\begin{corollary}
Use the same notation as in Theorem $\ref{theorem:rep}$ and assume that all $G^{(m)}$'s are amenable.
Then, for every $\omega\in\beta{\mathbb N}$ there is a canonical $*$-iso\-mor\-phism 
\[
\pi_\omega(\mathrm{C}^*(\{ v_g \mid g\in F_k\})) \cong \mathrm{C}^*_{\mathrm{max}}\left[G^{(\omega)}\right].
\]
Moreover, it is the range of a conditional expectation on $\mathrm{C}^*_{\mathrm{u, max}}[X]_\omega$.
\end{corollary}
\begin{proof}
Since $G^{(m)}$'s are amenable, there are $G^{(m)}$-invariant 
states $\mu_m$ on $\ell_\infty\left(G^{(m)}\right)$, and hence for every $\omega\in\beta{\mathbb N}$ 
the state $\ell_\infty(X)\ni\xi\mapsto\lim_\omega\mu_m(\xi^{(m)})$ gives rise to 
a $G^{(\omega)}$-invariant state on $\ell_\infty(X)_\omega$. 
It follows that the canonical $*$-homo\-mor\-phism from 
$\mathrm{C}^*_{\mathrm{max}}\left[G^{(\omega)}\right]$ into 
$\ell_\infty(X)_\omega\rtimes_{\max}G^{(\omega)}$ is faithful.
\end{proof}

This corollary gives an alternative proof that if 
$X=\bigsqcup_m\mathrm{Cay}\left(G^{(m)},s_1^{(m)},\ldots,s_k^{(m)}\right)$ 
has geometric property $(\mathrm{T})$, 
then all the groups in the Cayley boundary of $\left\{ G^{(m)} \right\}_m$ have property $(\mathrm{T})$.
It also recovers a result of \cite{MimuraSakoI} 
that $X$ has property $\mathrm{A}$ if and only if all the groups in 
the Cayley boundary are amenable. 

On this occasion, we clarify the relation between the two different notions 
of a disjoint union and their associated maximal uniform Roe algebras. 
So, let $\bigsqcup_m X^{(m)}$ (resp.\ $\coprod_m X^{(m)}$) be 
the disjoint union (resp.\ coarse disjoint union) of a given sequence 
$\left\{X^{(m)}\right\}_m$ of metric spaces. 
Thus, as a set $X:=\bigsqcup_m X^{(m)}=\coprod_m X^{(m)}$ and 
as a coarse space 
${\mathbb C}_{\mathrm u}\left[\coprod_m X^{(m)}\right]
 = {\mathbb C}_{\mathrm u}\left[\bigsqcup_m X^{(m)}\right]+{\mathbb K}_0$, 
where ${\mathbb K}_0$ is the algebra of the finitely supported kernels on $X$. 
Since ${\mathbb K}_0$ is an ideal of ${\mathbb C}_{\mathrm u}\left[\coprod_m X^{(m)}\right]$ 
and it has the unique $\mathrm{C}^*$-com\-ple\-tion ${\mathbb K}(\ell_2(X))$, 
it gives rise to an embedding 
${\mathbb K}(\ell_2(X))\hookrightarrow \mathrm{C}^*_{\mathrm{u, max}}\left[\coprod_m X^{(m)}\right]$ 
as a closed two-sided ideal. 
There is a conditional expectation $E$ from 
${\mathbb C}_{\mathrm u}\left[\coprod_m X^{(m)}\right]$ onto 
${\mathbb C}_{\mathrm u}\left[\bigsqcup_m X^{(m)}\right]$, given by 
$E(a)_{x,y}=a_{x,y}$ if $x$ and $y$ belong to the same $X^{(m)}$ else $E(a)_{x,y}=0$. 

\begin{proposition}
Let $\left\{X^{(m)}\right\}_m$ be a sequence of metric spaces. Then, the 
canonical embedding of 
${\mathbb C}_{\mathrm u}\left[\bigsqcup_m X^{(m)}\right]$ into 
${\mathbb C}_{\mathrm u}\left[\coprod_m X^{(m)}\right]$
extends to a faithful embedding of 
$\mathrm{C}^*_{\mathrm{u, max}}\left[\bigsqcup_m X^{(m)}\right]$ into 
$\mathrm{C}^*_{\mathrm{u, max}}\left[\coprod_m X^{(m)}\right]$.
Hence, one has a canonical identification
\[
\mathrm{C}^*_{\mathrm{u, max}}\left[\coprod_m X^{(m)}\right]
=\mathrm{C}^*_{\mathrm{u, max}}\left[\bigsqcup_m X^{(m)}\right]+{\mathbb K}(\ell_2(X)).
\]
The conditional expectation $E$ also extends to a faithful, unital and 
completely positive conditional expectation from 
$\mathrm{C}^*_{\mathrm{u, max}}\left[\coprod_m X^{(m)}\right]$ onto 
$\mathrm{C}^*_{\mathrm{u, max}}\left[\bigsqcup_m X^{(m)}\right]$. 
\end{proposition}
\begin{proof}
We first note that $J_0:={\mathbb K}_0\cap {\mathbb C}_{\mathrm u}\left[\bigsqcup_m X^{(m)}\right]$ 
has unique $\mathrm{C}^*$-com\-ple\-tion $J=\bigoplus_m{\mathbb M}_{X^{(m)}}$ 
(the $\mathrm{C}^*$-direct sum). Since 
${\mathbb C}_{\mathrm u}\left[\coprod_m X^{(m)}\right]/{\mathbb K}_0
 ={\mathbb C}_{\mathrm u}\left[\bigsqcup_m X^{(m)}\right]/J_0$, one has 
the commuting diagram 
\[\begin{CD}
J @>>> \mathrm{C}^*_{\mathrm{u, max}}\left[\bigsqcup_m X^{(m)}\right] @>>> \mathrm{C}^*_{\mathrm{u, max}}\left[\bigsqcup_m X^{(m)}\right]/J\\
@VVV @V{\iota}VV @V{\cong}VV \\ 
{\mathbb K}(\ell_2(X)) @>>> \mathrm{C}^*_{\mathrm{u, max}}\left[\coprod_m X^{(m)}\right] @>>> \mathrm{C}^*_{\mathrm{u, max}}\left[\coprod_m X^{(m)}\right]/{\mathbb K}(\ell_2(X))
\end{CD}\]
where the right column morphism is an isomorphism. 
Since the inclusion $\iota$ is injective on $J$, 
it is injective on $\mathrm{C}^*_{\mathrm{u, max}}\left[\bigsqcup_m X^{(m)}\right]$ also.
That $E$ extends to a unital and completely positive map on 
$\mathrm{C}^*_{\mathrm{u, max}}\left[\coprod_m X^{(m)}\right]$ follows from the fact that 
for every $a \in {\mathbb C}_{\mathrm u}\left[\coprod_m X^{(m)}\right]$ one has 
\[
E(a) = \sum_{m=1}^N 1_{X^{(m)}} a 1_{X^{(m)}} + (1-\sum_{m=1}^N 1_{X^{(m)}}) a (1-\sum_{m=1}^N 1_{X^{(m)}})
\]
for a large enough $N$, and in particular $\|E(a)\|_{\mathrm{max}}\le\|a\|_{\mathrm{max}}$.
Since $E$ is faithful on ${\mathbb K}(\ell_2(X))$ and 
$\mathrm{C}^*_{\mathrm{u, max}}\left[\coprod_m X^{(m)}\right]/{\mathbb K}(\ell_2(X))$, 
it is faithful on $\mathrm{C}^*_{\mathrm{u, max}}\left[\coprod_m X^{(m)}\right]$.
\end{proof}

\bibliographystyle{amsalpha}
\bibliography{cayley3.bib}

\end{document}